\documentclass[reqno]{amsart}

\usepackage{amsmath}
\usepackage{amsfonts}
\usepackage{amssymb}
\usepackage{graphicx}
\usepackage[shortlabels]{enumitem}
\usepackage{xcolor}
\usepackage{IEEEtrantoolsNEW}

\setlist{leftmargin=*}
\setlist[itemize]{labelindent=\parindent}

\newtheorem{theorem}{Theorem}[section]
\newtheorem{lemma}[theorem]{Lemma}
\newtheorem{proposition}[theorem]{Proposition}
\newtheorem*{thmRead}{Theorem R}
\newtheorem*{thmGodInd}{Theorem GI}

\newtheorem{observations}[theorem]{Observations}
\newtheorem{corollary}[theorem]{Corollary}

\newtheorem{fact}[theorem]{Fact}

\theoremstyle{definition}
\newtheorem{definition}{Definition}

\newtheorem{remark}[theorem]{Remark}
\newtheorem{problem}[theorem]{Problem}
\newtheorem*{ack}{Acknowledgment}
\newtheorem{example}[theorem]{Example}

\def\en{\mathbb N}
\def\qe{\mathbb Q}
\def\er{\mathbb R}

\def\J{\mathcal J}

\newcommand{\diam}{\operatorname{diam}}
\newcommand{\supp}{\operatorname{supp}}
\newcommand{\Int}{\operatorname{Int}}
\newcommand{\conv}{\operatorname{conv}}

\newcommand{\Ker}{\operatorname{Ker}}
\newcommand{\att}{\operatorname{att}}

\newcommand{\sspan}{\operatorname{span}}
\newcommand{\dist}{\operatorname{dist}}
\newcommand{\NA}{\operatorname{NA}}
\newcommand\setsep{;\ }
\newcommand{\vertiii}[1]{{\left\vert\kern-0.25ex\left\vert\kern-0.25ex\left\vert #1 
    \right\vert\kern-0.25ex\right\vert\kern-0.25ex\right\vert}}
		
\def\cconv{\overline{\conv}}
\def\sconv{\conv_{\sigma}}

\newcounter{pfparcount}
\def\pfpar{\stepcounter{pfparcount}{\bf\S \arabic{pfparcount}: }}
\def\pfparref{\stepcounter{pfparcount}{\bf\S \arabic{pfparcount}}}
\def\pfparcite#1{{\bf\S #1}}

\begin{document}
\title[Norm-attaining functionals and proximinal subspaces]{Norm-attaining functionals and proximinal subspaces}
\author{Martin Rmoutil}

\thanks{The research leading to these results has received funding from the Grant GA\v{C}R P201/12/0290, and the European Research  Council under the European Union's Seventh Framework Programme (FP/2007-2013) / ERC Grant Agreement n.291497.}
\email[M.~Rmoutil]{martin@rmoutil.eu}
\address{University of Warwick, Mathematics Institute, Zeeman Building, Coventry CV4 7AL, United Kingdom}
\address{Charles University, Faculty of Mathematics and Physics, Department of\linebreak Mathematical Analysis, Sokolovsk\'a 83, 186 75 Prague 8, Czech Republic}

\date{\today}
\begin{abstract}
G.~Godefroy asked whether, on any Banach space, the set of norm-attaining functionals contains a $2$-dimensional linear subspace. We prove that a recent construction due to C.J.~Read provides an example of a space which does not have this property.

This is done through a study of the relation between the following two sentences where $X$ is a Banach space and $Y$ is a closed subspace of finite codimension in $X$:
\begin{enumerate}[\bf (A)]
\item $Y$ is proximinal in $X$.
\item $Y^\bot$ consists of norm-attaining functionals.
\end{enumerate}
We prove that these are equivalent if $X$ is the Read's space.

Moreover, we prove that any non-reflexive Banach space $X$ with any given closed subspace $Y$ of finite codimension at least $2$ admits an equivalent norm such that implication {\bf (B)}~$\Longrightarrow$~{\bf (A)}  fails.
\end{abstract}

\maketitle

\section{Introduction}
Let $X=(X,\|\cdot\|)$ be a normed linear space. Then by $X^*$ we denote the topological dual of $X$ and by $\NA(X)=\NA(X,\|\cdot\|)$ the set of all functionals in $X^*$ for which there exists $x\in S_X$ such that $x^*(x)=\|x\|$; its elements are \emph{norm-attaining functionals}.

A set $F\subset X$ is said to be \emph{proximinal} if for each $x\in X$ there is $y\in F$ with $\| x-y \|=\dist (x,F)$. Note that proximinal sets are closed.

In this paper we investigate the relation of the following two sentences in which $X$ is a Banach space and $Y\subset X$ is a closed subspace:
\begin{enumerate}[\bf (A)]
\item $Y$ is proximinal in $X$.
\item $Y^\bot:=\{ x^*\in X^* \setsep Y\subset \Ker (x^*)\} \subset \NA(X)$.
\end{enumerate}

In view of the easy observation that a functional $x^*\in X^*$ is norm-attaining if and only if $\Ker(x^*)$ is proximinal (see e.g. Corollary~\ref{C:proxNA}), it is reasonable to ask whether sentences {\bf (A)} and {\bf (B)} are also equivalent when the codimension of $Y$ in $X$ is higher than $1$. And, indeed, implication $\text{\bf (A)}\Longrightarrow\text{\bf (B)}$ is true whenever the quotient space $X/Y$ is reflexive (cf. Lemma~\ref{L:proxchar}). In particular, it holds when the codimension of $Y$ in $X$ is finite (and therefore $\dim(X/Y)<\infty$) which is the case we are interested in the most.

However, implication $\text{\bf (B)}\Longrightarrow\text{\bf (A)}$ is not necessarily true even if $Y$ has finite codimension. The validity of this implication (and related statements) for various spaces has been studied by several authors (e.g. A.L.~Garkavi, V.~Indumathi or G.~Godefroy) and counterexamples are known: e.g. Indumathi \cite[Section 4.]{Ind82} shows that in any infinite-dimensional space $L_1(T,\nu)$ (here $(T,\nu)$ is a positive measure space) there exists a closed non-proximinal subspace $M$ of finite codimension $d\geq 2$ such that all subspaces $N\supsetneq M$ are proximinal; in particular, $L_1(T,\nu)$ fails implication $\text{\bf (B)}\Longrightarrow\text{\bf (A)}$.  In another work \cite{Ind87} of Indumathi, the author studies the relation of {\bf (A)} and {\bf (B)} (and similar statements) in more detail and defines a space to be $R(1)$ if implication $\text{\bf (B)}\Longrightarrow\text{\bf (A)}$ is true for any closed subspace $Y\subset X$ of finite codimension (in general, a non-reflexive space is said to be $R(n)$ \emph{space} if proximinality of its closed subspaces of finite codimension at least $n$ is implied by the proximinality of all $n$-codimensional subspaces containing it). 

In Section~5 we show that, given any non-reflexive Banach space $X$ and its closed subspace $Y$ of finite codimension at least $2$, there is an equivalent norm $\vertiii{\cdot}$ on $X$ such that $Y$ is not proximinal in $(X,\vertiii{\cdot})$, but $Y^\bot\subset \NA(X,\vertiii{\cdot})$ (see Theorem~\ref{C:counterex}). In particular, any non-reflexive Banach space can be renormed so that it is not $R(1)$.

The main motivation for the present work is the following Problem of G.~Godefroy together with a recent theorem of C.J.~Read below.
\begin{problem}[{\cite[Problem III.]{GodBSc0}}, also e.g. \cite{BanGod06} or \cite{AcAiArGa07}] \label{Problem1}
Does the set $\NA(X)$ contain a $2$-dimensional subspace for any Banach space $X$?
\end{problem}

The well-known theorem of E.~Bishop and R.R.~Phelps \cite{BiPh61} states that for any Banach space $X$, the set $\NA(X)$ is dense in $X^*$; in particular, $\NA(X)$ is nonempty and so it clearly contains $1$-dimensional subspaces. Since then, a number of papers studying the structure of $\NA(X)$ from various viewpoints have been written; the ones which are interesting for us are concerned with the linear structure of $\NA(X)$, in particular the question whether $\NA(X)$ is lineable or spaceable (for definition see Section~2). For example, in \cite{BanGod06} it is shown that if $X$ is an Asplund space enjoying the Dunford-Pettis property, then $\NA(X)$ is not spaceable (e.g. $C(K)$ with $K$ scattered (Hausdorff) compactum; however, $\NA(C(K))$ for $K$ infinite is always lineable---see \cite{AcAiArGa07}). On the other hand, in \cite{AcAiArGa07} the authors prove that if the Banach space $X$ possesses an infinite-dimensional complemented subspace with a Schauder basis, then $X$ can be equivalently renormed so that $\NA(X)$ is lineable. Some articles also investigate the lineability of $X^*\setminus\NA(X)$, e.g. \cite{AcAiArGa07} or \cite{Gar08}. In \cite{Gar08} it is observed that if $\J$ denotes the James space, then $\J^*\setminus \NA(\J)$ is not even $2$-lineable (i.e. $(\J^*\setminus \NA(\J))\cup \{0\}$ contains no 2-dimensional subspace). 

A related problem (again \cite[Problem III.]{GodBSc0}) was to decide whether every Banach space contains a proximinal subspace of codimension $2$. This problem was recently solved by C.J.~Read:

\begin{thmRead}[C.J.~Read \cite{Read}] \label{T:Read}
Let $\vertiii{\cdot}$ be the renorming of $c_0$ from Definition~\ref{defRead}. Then $(c_0,\vertiii{\cdot})$ contains no proximinal subspaces of finite codimension $2$ or larger.
\end{thmRead}
In Section~4 we show that this space also works as a counterexample for Problem~\ref{Problem1}. More precisely, our Theorem~\ref{T:ReadWorks} states that $\NA(c_0,\vertiii{\cdot})$ contains no $2$-dimensional subspaces. This is done by showing that $(c_0,\vertiii{\cdot})$ is $R(1)$ (i.e. implication $\text{\bf (B)}\Longrightarrow\text{\bf (A)}$ holds for it) and then applying Read's result. 
Among the known sufficient conditions for a Banach space $X$ to be $R(1)$ are the following three: 
\begin{itemize}
\item the dual space $X^*$ is (G\^ateaux) smooth \cite[Corollary 2]{Ind87};
\item $X$ is WLUR (for definition see Section 2.) \cite[Corollary 3]{Ind87};
\item $X$ is a subspace of $c_0$ \cite{GodInd99}.
\end{itemize}
None of these seems to be suitable in the case of $(c_0,\vertiii{\cdot})$, however; our proof has to be more direct, although it is done similarly to our new proof of the sufficiency of $X$ being WLUR (see Corollary \ref{C:suffWLUR}).

The essential (but simple) tool we use throughout the paper is a characterization of conditions {\bf (A)} and {\bf (B)} in terms of the quotient space $X/Y$ (see parts \emph{(i)} and \emph{(ii)} of Lemma~\ref{L:proxchar}); for example, a closed subspace $Y\subset X$ is proximinal if and only if the quotient mapping $q:X\to X/Y$ maps $B_X$ onto the whole $B_{X/Y}$ (where $B_Z$ denotes the closed unit ball in $Z$), and we have a similar condition characterizing the fact that $Y^\bot\subset \NA(X)$. It is not difficult to see that our characterization of proximinality is equivalent to that of Garkavi \cite{Gar63} (see \cite[Theorem A]{Ind87}), which is not formulated using the quotient space. The characterization of Garkavi was used in many subsequent articles on the topic, but it seems that our approach using quotient spaces is, at least in some situations, more natural.

In Lemma~\ref{L:proxchar}~\emph{(iv)} we use our characterizations of {\bf (A)} and {\bf (B)} to prove that implication $\text{\bf (B)}\Longrightarrow\text{\bf (A)}$ holds whenever $X/Y$ is strictly convex. Moreover, in Proposition~\ref{P:propSC} we show that this is the case whenever the norm on $X$ is WLUR (thus obtaining a different proof of this known sufficient condition for $R(1)$). This is not completely trivial since in Example~\ref{E:SCquot} we have a strictly convex Banach space $X$ and its closed subspace $Y$ of codimension $2$ such that $X/Y$ is not strictly convex.

The proof of Theorem~\ref{T:ReadWorks} (i.e. that Read's space solves Problem~\ref{Problem1}) is done by contradiction: Setting $X=(c_0,\vertiii{\cdot})$ we assume that there is a closed subspace $Y\subset X$ of finite codimension at least $2$ such that $Y^\bot\subset\NA(X)$, and then we prove that $X/Y$ is strictly convex. It follows from Lemma~\ref{L:proxchar} that $Y$ is proximinal which is a contradiction with Theorem~R.

We also use quotient spaces to obtain our renorming of a non-reflexive space $X$ such that the resulting space contradicts implication $\text{\bf (B)}\Longrightarrow\text{\bf (A)}$. Indeed, in Lemma~\ref{T:Xgen} we take any closed subspace $Y\subset X$ of finite codimension in $X$. Given a set $A\subset X/Y$ such that $U_{X/Y}\subset A \subset  B_{X/Y}$ satisfying some technical conditions, we find a renorming of $X$ such that $q(B_X)=A$. Then, using our Lemma~\ref{L:proxchar} we show in Theorem~\ref{C:counterex} that if the codimension of $Y$ is at least $2$, then $A$ can be chosen in such a way that the resulting renorming satisfies {\bf (B)}, but not {\bf (A)}.

\section{Notation and Definitions}
Let $X=(X,\|\cdot\|)$ be a normed linear space over the field of real numbers. Then $0$ is the origin in $X$ (there is no risk of confusion), $X^*$ is its (topological) dual, $B_X$ its closed unit ball, $U_X$ is the open unit ball and $S_X$ denotes the unit sphere. For a functional $x^*\in X^*$ we use the notation $\Ker(x^*):=(x^*)^{-1}(0)$ and $\att (x^*):=\{ x\in B_X\setsep x^*(x)=\|x^*\| \}$. As explained in the Introduction, $\NA(X)=\{x^*\in X^*\setsep \att(x^*)\neq\emptyset\}$. If $Y$ is a normed linear space, by $X\cong Y$ we mean that $X$ is linearly isomorphic to $Y$, and $X\oplus_\infty Y$ is the direct sum of $X$ and $Y$ equipped with the maximum norm, i.e. for $(x,y)\in X \oplus_\infty Y$ we have $\|(x,y)\| = \max\{\|x\|, \|y\|\}$.

If $Y$ is a subspace of $X$, we write $Y^\bot$ for its annihilator $\{ x^*\in X^* \setsep Y\subset \Ker(x^*) \}$. Recall that the quotient space $X/Y$ consists of all cosets of the form $x+Y$ and is equipped with the norm $\|x+Y\|=\dist(0,x+Y)$. If $X$ is a Banach space and $Y$ is closed, then $X/Y$ is a Banach space and its dimension is equal to the codimension of $Y$ in $X$. We shall often work with the quotient mapping $q:X\to X/Y:x\mapsto x+Y$, which is a linear operator satisfying $\|q\|=1$. 

Next, if $A\subset X$ is any set, then $\Int(A)$, $\overline{A}$ and $\partial A$ stand for the interior,  closure and boundary of $A$ respectively. By $x_n\stackrel{w}{\to}x$ we mean that $x_n$ converges to $x$ weakly in $X$, and similarly for $w^*$. If we do not specify which topology we mean, it is usually the norm topology; e.g. $x_n\to x$ means $\lim_{n\to\infty} \|x_n - x\|=0$.

We recall that the norm $\|\cdot\|$ on $X$ is called \emph{strictly convex} (or \emph{rotund}; SC for short) if all the points of $S_X$ are extremal points of $B_X$, that is, if there are no nontrivial line segments contained in $S_X$. In that case $X$ is also called strictly convex.

The norm $\|\cdot\|$ is said to be \emph{weakly locally uniformly rotund} (WLUR for short) if $x_n\stackrel{w}{\rightarrow}x$ whenever $x\in X$ and $(x_n)_{n=1}^\infty \subset X$ are such that $\lim_{n\to\infty} \|x_n\|=\|x\|$ and $\lim_{n\to\infty}\|x+x_n\|= 2\|x\|$. Then the space $X$ is called WLUR. 

For any sequence $a\in\ell_\infty$ we denote $a=\big(a^j\big)_{j=1}^\infty=\big(a^1, a^2, \dots\big)$. 

Let $A$ be a bounded set in a Banach space. We denote by $\sconv(A)$ the $\sigma$-convex hull of $A$, that is
$$\sconv(A)=\Big\{ \sum_{i=1}^\infty \lambda_i x_i \setsep (\lambda_i)_{i=1}^\infty \subset [0,1],\; \sum_{i=1}^\infty \lambda_i =1 \text{ and }(x_i)_{i=1}^\infty \subset A \Big\}.$$

A subset $M$ of a Banach space is said to be:
\begin{itemize}
\item \emph{$n$-lineable} if $M\cup\{0\}$ contains an $n$-dimensional subspace;
\item \emph{lineable} if $M\cup\{0\}$ contains an infinite-dimensional subspace;
\item \emph{spaceable} if $M\cup\{0\}$ contains an infinite-dimensional closed subspace.
\end{itemize}
A comprehensive survey on linear subspaces of various sets is \cite{BePeSe14}.

\section{Observations and lemmas}
The following observations are obvious and will be applied without reference.	
\begin{observations}\label{O:fakta} Let $X$ be a normed linear space.
\begin{enumerate}[(i)]
\item Any proximinal set in $X$ is closed.
\item Let $F$ be a proximinal set in $X$ and let $x\in X$. Then $x+F$ is also proximinal in $X$.
\item A subspace $Y\subset X$ is proximinal if and only if for any $x\in X$ there is a point $z$ with minimal norm in the coset $x+Y$.
\end{enumerate}
\end{observations}

The next facts are well-known; we include them for the sake of completeness.
\begin{fact} \label{F:attsspan}
\mbox{}
\begin{enumerate}[(i)]
\item Let $X$ be a normed linear space and $y^*, x_1^*,\dots, x_n^*\in X^*$. Then 
\linebreak
$y^*\in\sspan\{x_1^*, \dots, x_n^*\}$ if and only if $\bigcap_{i=1}^n \Ker(x_i^*)\subset \Ker (y^*)$.
\item Let $Z$ be any normed linear space. Then $Z$ is not strictly convex if and only if there is a functional $\varphi \in Z^*$ such that $\att(\varphi)$ contains more than one point.
\item 
Let $X$ be a Banach space, $Y\subset X$ be a closed subspace and $q:X\to X/Y$ the quotient mapping. Then $q^*:(X/Y)^*\to Y^\bot$ is an isometry onto $Y^\bot$. 
\end{enumerate}
\end{fact}

\begin{proof}
\emph{(i)} and \emph{(iii)} are standard; for proof of \emph{(iii)} see e.g. \cite[Proposition~2.6]{FHHMZ}.

\emph{(ii)}: Assume that $Z$ is not strictly convex. Take two distinct points $u,v\in S_Z$ such that for any $\lambda\in (0,1)$, $\lambda u+ (1-\lambda) v \in S_Z$, and find (using the Hahn--Banach theorem) a functional $\varphi\in S_{Z^*}$ such that $\varphi(\frac{u+v}{2})=1$. Clearly, if $\varphi(u)<1$, then $\varphi(v)>1$, and vice versa. But $\| \varphi \| =1$, whence $\varphi(u)=\varphi(v)=1$. The opposite implication is trivial.
\end{proof}

The following lemma translates properties {\bf (A)} and {\bf (B)} to the language of quotient spaces, and is essential in the sequel.

\begin{lemma}\label{L:proxchar}
Let $X$ be a Banach space, $Y\subset X$ be a closed subspace and $q:X\to X/Y$ be the quotient mapping. Then:
\begin{enumerate}[(i)]
\item $Y \text{ is proximinal in }X \;\Longleftrightarrow\; q(B_X)=B_{X/Y}$.
\item $Y^\bot \subset \NA(X)\;  \Longleftrightarrow \; (\forall \varphi\in (X/Y)^*) \: (\exists u\in q(B_X)) \; \varphi(u)=\|\varphi\|$.
\item If $X/Y$ is reflexive and $Y$ is proximinal, then $Y^\bot \subset \NA(X)$.
\item If $X/Y$ is strictly convex and $Y^\bot \subset \NA(X)$, then $Y$ is proximinal.
\end{enumerate}
\end{lemma}

\begin{proof}
\emph{(i)}: Assume $Y$ is proximinal. Obviously $q(B_X)\subset B_{X/Y}$; to prove the other inclusion, take any $u\in B_{X/Y}$. Further, pick $z\in q^{-1}(u)$. Then $\dist(0,z+Y)=\|u\|\leq 1$, and by Observation~\ref{O:fakta}~(iii) there is a point $y\in z+Y$ such that $\|y\|=\dist(0,z+Y)$, so $q(y)=u$ and $y\in B_X$.

To prove the opposite implication, assume that $q(B_X)=B_{X/Y}$ and take an arbitrary $x\in X$. By homogeneity we can clearly assume that $\dist(0, x+Y)=1$, so $q(x)\in B_{X/Y}$, and our assumption now implies that there exists $y\in B_X$ such that $q(y)=q(x)$ (which is equivalent to $y\in x+Y$). By Observation~\ref{O:fakta}~(iii) we are done. 

\emph{(ii)}: Let $Y^\bot \subset \NA(X)$ and take an arbitrary $\varphi\in (X/Y)^*$. Then $y^*:=\varphi\circ q \in Y^\bot$, and so there exists $x\in B_X$ such that $\varphi(q(x))=y^*(x)=\|y^*\|=\|\varphi\|$ (by Fact~\ref{F:attsspan}~(iii)). Hence $q(x)\in q(B_X)\cap \att(\varphi)$. 

Conversely, take any $y^* \in Y^\bot$ and the corresponding $\varphi\in(X/Y)^*$ such that $y^*=\varphi\circ q$. By the assumption there exists $u\in q(B_X)\cap \att(\varphi)$. So there also exists $x\in B_X\cap q^{-1}(u)$. Then $y^*(x)=\varphi(q(x))=\varphi(u)=\|\varphi\|=\|y^*\|$, and $y^*$ is therefore norm-attaining.

\emph{(iii)}: Trivial from \emph{(i)}, \emph{(ii)} and the fact that on a reflexive space, all functionals are norm-attaining.

\emph{(iv)}: Assume the condition in \emph{(ii)} is satisfied. Clearly $U_{X/Y}\subset q(B_X)\subset B_{X/Y}$, so by \emph{(i)} it suffices to show that $S_{X/Y}\subset q(B_X)$. To that end, take an arbitrary $u\in S_{X/Y}$ and use the Hahn-Banach theorem to find $\varphi\in S_{(X/Y)^*}$ such that $u\in\att(\varphi)$. Since $X/Y$ is strictly convex, it follows by Fact~\ref{F:attsspan} that $\att(\varphi)=\{u\}$, and the condition in \emph{(ii)} now implies that $u\in q(B_X)$. 
\end{proof}

\begin{corollary}\label{C:proxNA}
Let $X$ be a Banach space. Then:
\begin{enumerate}[(a)]
\item Let $Y\subset X$ be a closed subspace of finite codimension in $X$. If $Y$ is proximinal, then $Y^\bot\subset \NA(X)$.
\item A functional $x^*\in X^*$ is norm-attaining if and only if $\Ker(x^*)$ is proximinal.
\end{enumerate}
\end{corollary}

\begin{proof}
\emph{(a):} The quotient space $X/Y$ is finite-dimensional and in particular it is reflexive.

\emph{(b):} In addition to being reflexive, the quotient space $X/\Ker(x^*)$ is (for trivial reasons) also strictly convex, and the statement follows.
\end{proof}

\begin{proposition}\label{P:propSC}
Let $X$ be a Banach space with a WLUR norm, and $Y\subset X$ be its closed subspace such that $Y^\bot\subset \NA (X)$. Then $X/Y$ is strictly convex. 
\end{proposition}
\begin{proof}
Aiming for a contradiction, suppose that $X$ is WLUR and $X/Y$ is not strictly convex. Then we can find a functional $\varphi \in (X/Y)^*$ such that $\| \varphi \| = 1$ and $\att(\varphi)$ contains more than one point (Fact \ref{F:attsspan}~(ii)). 
Now, set $x^*=\varphi\circ q$ where $q:X\to X/Y$ is the quotient mapping. Fact \ref{F:attsspan}~(iii) gives us that $\|x^*\|=\|\varphi\|=1$ and $x^*\in Y^\bot$. As $Y^\bot\subset\NA(X)$, $x^*$ attains its norm at a point $x\in S_X$ and $\varphi (q(x))= x^*(x)=1$, whence $q(x)\in \att (\varphi)$.

Take any $u\in \att(\varphi)\setminus \{ q(x) \}$ and any sequence of elements $x_n\in u$ such that $\lim_{n\to\infty }\|x_n\| = 1= \|x\|$. Then $\lim_{n\to\infty}\|x+x_n\|=2$ since 
$$2=\varphi(q(x_n)+q(x))=x^*(x_n+x)\leq \|x_n+x\|\leq\|x_n\| + \|x\| \to 2, \quad n \to \infty.$$
Our assumption that $X$ is WLUR yields that $x_n$ converges to $x$ weakly in $X$. 

To obtain the desired contradiction, take any $\psi \in (X/Y)^*$ such that $\psi(q(x))\neq \psi(u)$ and set $y^*:=\psi\circ q\in X^*$. Then for all $n\in\en $ we have $y^*(x)=\psi(q(x))\neq \psi(u) = y^*(x_n)$, so $y^*(x_n)$ does not converge to $y^*(x)$ which is a contradiction.
\end{proof}

Let us note that these observations provide a new proof for the following result of Indumathi (see \cite[Corollary 3]{Ind87}).
\begin{corollary}\label{C:suffWLUR}
Let $X$ be a Banach space with a WLUR norm and $Y\subset X$ be its closed subspace of finite codimension. Then $Y$ is proximinal if and only if $Y^\bot\subset\NA(X)$. 
\end{corollary}

\begin{proof}
The statement follows immediately from Proposition~\ref{P:propSC}, Lemma~\ref{L:proxchar} and Corollary~\ref{C:proxNA}.
\end{proof}

The following is a simple example of a strictly convex space $X$ with a finite-codimensional closed subspace $Y$ such that $X/Y$ is not strictly convex.

\begin{example} \label{E:SCquot}
By $[0,\omega]$ we mean the compact interval of ordinals with the order topology. Let us define the Banach space $X$ as $C([0,\omega]\times\{0,1\})$ with the norm $\|\cdot\|$ defined by
$$\|x\|^2= \|x\|^2_{\infty} + \sum_{k=0}^\infty 2^{-k}\big( x(k,0)^2+x(k,1)^2 \big).$$
It is easy to check that $\|\cdot\|$ is indeed a norm on $C([0,\omega]\times\{0,1\})$, which is moreover equivalent to $\|\cdot\|_\infty$.  It is also clear that the dual to $(C([0,\omega]\times\{0,1\}),\|\cdot\|_\infty)$ is $(\ell_1([0,\omega]\times\{0,1\}),\|\cdot\|_1)$, so $X^*$ is isomorphic to the latter space.

Set $Y:=\{ x\in X\setsep x(\omega,0)=x(\omega,1)=0 \}$; then $Y$ is a closed subspace of codimension $2$ and $Y^\bot = \sspan\{ \delta_{(\omega,0)}, \delta_{(\omega,1)} \}$ and $Y$. We claim that $X$ is strictly convex, but $X/Y$ is not.

Let $x,y\in X$ be arbitrary points satisfying 
\begin{equation} \label{E:simplecharSC}
2\|x\|^2+2\|y\|^2 - \|x+y\|^2=0;
\end{equation}
by a well-known (and easy) characterization of strict convexity (see e.g. \cite{FHHMZ}) it suffices to prove that $x=y$. One can rewrite equation \eqref{E:simplecharSC} using the definition of $\|\cdot\|$, and observe that the left hand side is the sum of the term
$$2\|x\|^2_\infty+2\|y\|^2_\infty - \|x+y\|^2_\infty,$$
and all terms of the form
$$2^{-k}\big(2x(k,i)^2+2y(k,i)^2-(x(k,i)+y(k,i))^2\big), \quad (k,i)\in[0,\omega)\times\{0,1\}$$
which are easily seen to be non-negative. Since the sum of all these terms is equal to $0$, it follows that so is each of the terms. In particular, $x(k,i)=y(k,i)$ for all $(k,i)\in [0,\omega)\times\{0,1\}$, and the continuity of $x$ and $y$ now implies that $x=y$.

We shall now prove that $X/Y$ is isometrically isomorphic to $(\er^2,\|\cdot\|_\infty)$; in particular, $X/Y$ is not strictly convex. 
Consider the quotient mapping $q:X\to X/Y : x\mapsto x+Y = q(x)$. Obviously $z\in q(x)$ if and only if $z(\omega,i)=x(\omega,i)$, $i=0,1$, whence the mapping $L:X/Y\to \er^2: x+Y\mapsto (x(\omega,0), x(\omega,1))$ is a linear bijection. It remains to show that $L:X/Y\to \big( \er^2,\|\cdot\|_\infty \big)$ is an isometry.

Let $x\in X$ be given; by definition,  
$$\|x+Y\|=\dist(0,x+Y)\geq \max\{|x(\omega,0)|, |x(\omega,1)| \}.$$
To prove the opposite inequality, set for all $N\in\en$,
$$x_N:=x\cdot \chi_{[N,\omega]\times\{0,1\}},$$
where $\chi_A$ denotes the characteristic function of $A\subset[0,\omega]\times\{0,1\}$; then $x_N\in x+Y$ as $x_N(\omega,i)=x(\omega,i)$, $i=0,1$. Now,
$$\|x_N\|^2= \|x_N\|^2_\infty + \sum_{k=N}^\infty 2^{-k} \big( x(k,0)^2+x(k,1)^2 \big),$$
and it follows from the continuity of $x$ that 
$$\lim_{N\to\infty}\|x_N\|=\max\{|x(\omega,0)|,|x(\omega,1)| \};$$
hence, $\dist(0,x+Y)\leq \max\{|x(\omega,0)|, |x(\omega,1)| \}.$

\end{example}

In view of the last example and Proposition~\ref{P:propSC} it is natural to ask the following.
\begin{problem}
Can the assumption that $X$ be WLUR in Proposition~\ref{P:propSC} be weakened to SC? Equivalently, can Example~\ref{E:SCquot} be done in such a way that we moreover have $Y^\bot\subset\NA(X)$? 
\end{problem}

\begin{remark} The following classical results are related to the problem above. Let $X$ be a Banach space.
\begin{enumerate}[(a)]
\item Let $X$ be SC and $Y$ be a reflexive subspace of $X$. Then $X/Y$ is SC (V.~Klee, \cite{Klee59}).
\item The following are equivalent (see e.g. M.~Day, \cite[Chapter VII, \S2]{Day73}):
	\begin{enumerate}[(i)]
	\item $X^*$ is G\^ateaux smooth (equivalently, at each point of $S_{X^*}$ there is only one supporting hyperplane of $B_{X^*}$);
	\item every quotient of $X$ is SC;
	\item every $2$-dimensional quotient of $X$ is SC.
	\end{enumerate}
\end{enumerate}
\end{remark}

It is interesting to note that having a strictly convex norm on $X/Y$ is not the only known condition which ensures that the implication $\text{\bf (B)}\Longrightarrow\text{\bf (A)}$ holds. The following theorem can be found in \cite{GodInd99} or \cite{GodBSc0}:

\begin{thmGodInd}[G.~Godefroy, V.~Indumathi] \label{T:GIcnaught}
Let $X$ be a closed subspace of $c_0$ and $Y\subset X$ be a closed finite-codimensional subspace of $X$. Then $Y$ is proximinal in $X$ if and only if $Y^\bot\subset\NA(X)$. 
\end{thmGodInd}
\noindent Instead of showing that $X/Y$ is strictly convex (which is not the case), the authors essentially prove that $X/Y$ is polyhedral (i.e. the unit ball is the convex hull of finitely many points), and then show that each of the extremal points is contained in $q(B_X)$ where $q:X\to X/Y$ is the quotient mapping (in the same manner as we show it for each point of the sphere $S_{X/Y}$); it then follows from the convexity of $q(B_X)$ that $q(B_X)=B_{X/Y}$, and therefore $Y$ is proximinal. (In their proof, our Lemma~\ref{L:proxchar}~\emph{(i)} is replaced by an essentially equivalent theorem of Garkavi, which is not formulated in terms of quotient spaces.)

The next proposition can be regarded as a generalization of Corollary~\ref{C:proxNA}~\emph{(a)}. This can be seen if we reformulate the corollary as follows:

\emph{Let $X$ be a Banach space and $x^*\in X^*$. Then the hyperplane $\Ker(x^*)$ is proximinal whenever it contains a finite-codimensional proximinal subspace $Y$.}

\begin{proposition}\label{P:prox1}
Let $X$ be a normed linear space and let $Y\subset Z\subset X$ be linear subspaces such that the codimension of $Y$ in $Z$ is finite. If $Y$ is proximinal (in $X$), then $Z$ is also proximinal. 
\end{proposition}

\begin{proof}
First assume that the codimension of $Y$ in $Z$ is $1$. Fix a functional $x^*\in S_X$ such that $Y=Z\cap \Ker (x^*)$ and any vector $v\in Z\setminus Y$; then it is easy to see that 
$$Z=\bigcup_{c\in\er}(cv+Y).$$
Now, pick an arbitrary $x\in X$; we aim to find a point $y\in x+Z$ with $\|y\|=\dist(0, x+Z)$. To that end, let us consider the function $\varphi:\er\to\er$ defined at $c\in\er$ by the formula
$$\varphi(c):=\dist(0,x+cv+Y).$$

Assume, for the moment, that we have found a global minimum of $\varphi$ at a point $c$; then we can use the proximinality of $Y$ to find a point $y\in x + cv + Y$ with minimal norm. Clearly $y$ has the minimal norm also in $x+Z$. 

Since $\varphi$ is easily seen to be continuous (even $\|v\|$-Lipschitz), it is enough to show that there is a compact set $I\subset \er$ with $\varphi(\er\setminus I)\subset (\varphi(0),\infty)$. We claim that it suffices to set
$$I:=\{ c\in\er\setsep |x^*(x+cv)|\leq \|x\| \}.$$
Indeed, if $c\in\er\setminus I$, then $|x^*(x+cv)|> \|x\|$. Since $Y\subset \Ker (x^*)$, the functional $x^*\in S_{X^*}$ is constant on any coset of the form $z+Y$, $z\in X$. It follows that 
\begin{IEEEeqnarray*}{rCl}
\varphi(c)&=&\inf\{\|x+cv+y\|\setsep y\in Y\} \\
&\geq &\inf\{ |x^*(x+cv+y)|\setsep y\in Y\} = |x^*(x+cv)| > \|x\|\geq\varphi(0);
\end{IEEEeqnarray*}
In particular, $0\in I$. The boundedness of $I$ follows from the fact that $v\notin \Ker(x^*)$. 

Now assume that we have proved the assertion of the Proposition in case the codimension of $Y$ in $Z$ is equal to $n\in\en$, and suppose we have closed subspaces $Y$ and $Z$ of $X$ such that $Y\subset Z$ with codimension $n+1$. This means that there are some $x_1^*,\dots, x_{n+1}^*\in X^*$ such that $Y=Z\cap\bigcap_{i=1}^{n+1}\Ker (x_i^*)$. Set $Z_1:=Z\cap\bigcap_{i=1}^n \Ker(x_i^*)$. Then the codimension of $Y$ in $Z_1$ is $1$, so $Z_1$ is proximinal, and the codimension of $Z_1$ in $Z$ is $n$, so $Z$ is proximinal. 
\end{proof}

\begin{remark}
Following W.~Pollul \cite{Pollul}, by $Z \subset^p Y$ we mean that $Z$ is a proximinal subspace of the normed linear space $Y$; it is then natural to ask under what conditions is the relation $\subset^p$ transitive (in which case the ambient space is said to be a \emph{$P$-space}). Indumathi \cite[Corollary~5]{Ind87} gives a characterization: A normed linear space $X$ is a $P$-space if and only if it is an $R(1)$ space and $\NA(X)$ is orthogonally linear (for definition see the same article). In particular, $X$ is a $P$-space if it is an $R(1)$ space and $\NA(X)$ is a vector space. 

Another related question is whether proximinality of subspaces is preserved under intersections. This is not always the case either; if a Banach space $X$ has the property that the intersection of arbitrary two finite-codimensional proximinal subspaces is proximinal, $X$ is said to be \emph{lattice-proximinal}. A simple characterization was given by Godefroy and Indumathi in \cite{GodInd13}: $X$ is lattice-proximinal if and only if it is $R(1)$ and $\NA(X)$ is a vector space. An example of such a space is $c_0$; we already know that it is $R(1)$ (Theorem~GI), and it is easy to see that norm-attaining functionals on $c_0$ are exactly finitely supported elements of $\ell_1$.
\end{remark}

\section{$\NA(X)$ need not be 2-lineable}

The following renorming of $c_0$ is due to C.J.~Read \cite{Read} and it is the first Banach space known to have no proximinal subspaces of finite codimension strictly larger than $1$ (Theorem~R). In this section we will show that the set of norm-attaining functionals on this space is not $2$-lineable, solving Problem~\ref{Problem1}.

\begin{definition}[C.J.~Read]\label{defRead}
Let $c_{00}(\qe)$ be the terminating sequences with rational coefficients, and let $(u_n)_{n=1}^\infty$ be a sequence of elements of $c_{00}(\qe)$ which lists every element infinitely many times.

Further, let $(a_n)_{n=1}^\infty$ be a strictly increasing sequence of positive integers satisfying for each $n\in \en$,
$$a_n>\max \:\supp u_n, \quad \text{and}\quad a_n\geq \|u_n\|_1.$$
The norm $\vertiii \cdot$ on $c_0$ is defined as follows:
$$\vertiii x = \|x\|_\infty + \sum_{n=1}^\infty 2^{-a_n^2} |\langle x, u_n-e_{a_n}\rangle |.$$
Here $\|x\|_\infty =\sup_n |x_n|$ is the usual norm on $c_0$, $e_j$'s are the vectors from the usual cannonical basis, and the duality $\langle \cdot, \cdot \rangle$ is the $\langle c_0, l_1\rangle$ duality.
\end{definition}

\begin{remark}\mbox{}
Proposition \ref{P:prox1} implies that if a normed linear space $X$ has no proximinal subspace of codimension $2$, then it has no proximinal subspace of any finite codimension larger than $2$ either. Hence, in the proof of C.J.~Read it is enough to show that the space $(c_0,\vertiii{\cdot})$ constructed there has no proximinal subspace of codimension $2$ (he proves it at once for all codimensions, slightly complicating the proof).
\end{remark}

The following theorem shows that the construction of Read does, in fact, answer Problem~\ref{Problem1}.

\begin{theorem}\label{T:ReadWorks}
Let $X$ be the space $(c_0,\vertiii{\cdot})$ from Definition~\ref{defRead}. Then $\NA(X)$ contains no linear subspaces of dimension larger than one.
\end{theorem}
\begin{proof}
We shall denote by $\vertiii{\cdot}$ the norm on $\ell_1$ which is dual to $\vertiii{\cdot}$ on $c_0$. Assume that $Y\subsetneqq X$ is a closed subspace of finite codimension and such that $Y^\bot\subset \NA(X)$; we aim to prove that the codimension of $Y$ in $X$ is equal to $1$. It follows from Theorem~R and Lemma~\ref{L:proxchar} that it is enough to prove the strict convexity of $X/Y$; we shall do so by contradiction.  

To that end, we can (by Fact~\ref{F:attsspan}) assume that we have a functional $\varphi\in S_{(X/Y)^*}$ such that $\att(\varphi)$ contains more than one point; set $x^*:=\varphi\circ q$. Then $x^*\in Y^\bot\subset\NA(X)$, and so there is $x\in \att(x^*)$. Fact~\ref{F:attsspan} gives that $\vertiii{x^*}=\|\varphi\|=1$ and $q(x)\in \att(\varphi)$. 

Let us take any point $u\in \att(\varphi)\setminus \{ q(x)\}$ and choose any sequence $(x_n)_{n=1}^\infty\subset u \subset X$ such that $\lim_{n\to\infty}\vertiii{x_n}=1$ (recall that $\|u\|=1$). Then $(x_n)_{n=1}^\infty$ is a bounded sequence and we can without loss of generality assume that there is $\tilde{x}\in \ell_\infty$ such that $x_n\stackrel{w^*}{\rightarrow} \tilde{x}$ in $\ell_\infty$. 

Now, 
$$2=\varphi(q(x)+q(x_n))=x^*(x+x_n)\leq \vertiii{x+x_n} \leq \vertiii{x}+\vertiii{x_n}\stackrel{n\to\infty}{\longrightarrow}2,$$
whence
$$\vertiii x + \vertiii{x_n} -\vertiii{x+x_n} \longrightarrow 0.$$
Since the norm $\vertiii{\cdot}$ is defined by a series of pseudonorms (which satisfy the triangle inequality), we in turn obtain for each $k\in\en$:
$$|\langle x, u_k-e_{a_k}\rangle|+|\langle x_n, u_k-e_{a_k}\rangle|-|\langle x+x_n, u_k - e_{a_k}\rangle |\stackrel{n\to\infty}{\longrightarrow} 0,$$
so, since $(x_n)$ converges to $\tilde{x}$ in weak$^*$ topology, we have that
\begin{equation}\label{E:samesign}
|\langle x, u_k-e_{a_k}\rangle|+|\langle \tilde{x}, u_k-e_{a_k}\rangle|-|\langle x+\tilde{x}, u_k - e_{a_k}\rangle |= 0.
\end{equation}
{\bf Claim:} There is $\lambda \geq 0$ such that $\tilde{x}=\lambda x$. In particular, $\tilde{x}\in X$ and $x_n\stackrel{w}{\rightarrow}\tilde{x}$.

Assume this is not the case. Then there is  $z^*\in c_{00}(\qe)\subset (\ell_1,\vertiii{\cdot})=X^*$ such that $z^*(x)<0<\tilde{x}(z^*)$, $|z^*(x)|>\|x\|_\infty$ and $\tilde{x}(z^*)>\|\tilde{x}\|_{\infty}$. 

Indeed, by the Hahn-Banach theorem we can find $z_1^*\in \ell_1$ separating the point $x$ and the weak$^*$-closed convex set $[0,\infty)\cdot \tilde{x}$, and $z_2^*\in\ell_1$ separating $\tilde{x}$ and the weak$^*$-closed convex set $[0,\infty)\cdot x$ with 
$$ z_1^*(x)<0\leq\tilde{x}(z_1^*), \quad \text{and} \quad z_2^*(x)\leq0 < \tilde{x}(z_2^*). $$
Setting $z_3^*:=z_1^*+z_2^*$, one can readily see that 
$$z_3^*(x) < 0 < \tilde{x}(z_3^*).$$
Using the density of $c_{00}(\qe)$ in $\ell_1$ we find a functional $z^*\in c_{00}(\qe)\subset\ell_1$ so close to $z_3^*$ that, again, $z^*(x)<0<\tilde{x}(z^*)$. Without loss of generality we can now assume that $|z^*(x)|>\|x\|_\infty$ and $\tilde{x}(z^*)> \|\tilde{x}\|_{\infty}$ because in the opposite case we would simply multiply $z^*$ by a sufficiently large rational constant.

Now, take $k\in\en$ such that $u_k=z^*$ (see Definition~\ref{defRead}); Then
\begin{align*}
\langle x, u_k-e_{a_k}\rangle &= z^* (x) - x^{a_k}\leq z^*(x) + \|x\|_\infty<0 \quad\text{ and}\\
\langle \tilde{x}, u_k- e_{a_k}\rangle &= \tilde{x}(z^*)-\tilde{x}^{a_k}\geq \tilde{x}(z^*)-\|\tilde{x}\|_\infty>0.
\end{align*}
But on the other hand it follows from \eqref{E:samesign} that for each natural $k$ we have 
$$\langle x, u_k - e_{a_k} \rangle \cdot \langle \tilde{x}, u_k-e_{a_k}\rangle \geq 0,$$
a contradiction which proves the Claim.

Finally, find a functional $\psi\in (X/Y)^*$ such that $\psi(q(x))<0<\psi(u)$ (we can do so e.g. using the same idea as with $z_3^*$ above in this proof). But then, setting $y^*:=\psi \circ q$, we have that 
$$y^*(\tilde{x})=\lambda y^*(x)=\lambda \psi(q(x))\leq 0<\psi(u)=y^*(x_n)$$
which contradicts the fact that $x_n\stackrel{w}{\to}\tilde{x}$, concluding the proof.
\end{proof}

\section{Non-$R(1)$ renormings}
The main result of the present section is Theorem~\ref{C:counterex}, which shows that the implication $\text{\bf (B)}\Longrightarrow\text{\bf (A)}$ (see the Introduction) can fail in all non-reflexive spaces if we take a suitable renorming; the purpose of all the preceding lemmas is to prove this theorem. At the end of this section we also prove a slightly stronger version of the key Lemma~\ref{T:Xgen} which holds for the space $\ell_1$.

\begin{lemma}\label{L:convint}
Let $x_1, x_2, \dots, x_n\in\er^d$ be such that $C:=\conv\{x_1, x_2, \dots, x_n\}$ has nonempty interior. Let $x=\sum_{i=1}^n \lambda_i x_i$ with $\sum_{i=1}^n \lambda_i=1$ and $\lambda_i >0$ for each $i$. Then $x\in\Int(C)$.
\end{lemma}
\begin{proof}
Assume that $x\in\partial C$ and take any non-zero functional $\varphi\in (\er^d)^*$ such that $\sup_C\varphi=\varphi (x)$ (it exists due to the Hahn-Banach Theorem). The hyperplane $H:=\varphi^{-1}(\varphi(x))$ is nowhere dense, so there is a $j\in\{1,\dots, n\}$ such that $x_j\in C\setminus H$ (otherwise $C\subset H$ which is impossible). But $\lambda_j>0$, so $\lambda_j \varphi(x_j)<\lambda_j \varphi(x)$, whence 
$$\varphi(x)=\varphi\Big(\sum_{i=1}^n \lambda_i x_i\Big) = \sum_{i\neq j}\lambda_i\varphi(x_i)+\lambda_j\varphi(x_j)<\sum_{i\neq j}\lambda_i \varphi(x)+\lambda_j \varphi(x)=\varphi(x),$$
which is a contradiction.
\end{proof}

\begin{lemma}\label{L:intbound}
If $C$ is a convex set in a normed linear space $X$, $y\in \Int(C)$ and $z\in \overline{C}$, then $\lambda y + (1-\lambda)z\in \Int(C)$ for any $\lambda \in (0,1)$.
\end{lemma}
\begin{proof}
If $C,y,z$ and $\lambda$ are as above, we can find $r>0$ such that $U(y,r)\subset C$ and a point $x\in C\cap U(z,\frac{\lambda r}{1-\lambda})$. Then the ball $U:=\lambda U(y,r) + (1-\lambda) x = U(\lambda y + (1-\lambda)x, \lambda r)$ is contained in $C$ as it consists of convex combinations of points of $U(y, r)$ and $x$. But $\|\lambda y + (1-\lambda)z -(\lambda y + (1-\lambda)x)\|=(1-\lambda)\|z-x\|<\lambda r$, and so $\lambda y + (1-\lambda)z \in U$. 
\end{proof}

\begin{proposition}\label{P:sconvisconv}
Let $A\subset \er^d$ be bounded. Then $\sconv (A)=\conv(A)$.
\end{proposition}

\begin{proof}
Take any $\sigma$-convex combination $x=\sum_{i=1}^\infty \lambda_i x_i$ of elements of $A$. If there are only finitely many non-zero $\lambda_i$'s, then $x\in \conv(A)$ trivially. So we can clearly assume that $\lambda_i\neq 0$ for every $i$, and we aim to prove that $x\in \conv\{x_i\setsep i\in \en \}\subset\conv(A)$.

Now, any convex set in $\er^d$ is either contained in a hyperplane, or has a nonempty interior. Thus, setting $C:=\conv \{x_i\setsep i\in \en \}$, we can assume that $\Int (C)\neq\emptyset$. Take $N\in \en$ so large that the interior of $C_N:=\conv\{x_i\setsep i=1,\dots, N\}$ is nonempty and set $\lambda:=\sum_{i=1}^N \lambda_i$; then by Lemma~\ref{L:convint}, $y:=\frac{1}{\lambda}\sum_{i=1}^N\lambda_i x_i \in \Int(C_N)\subset \Int(C)$. 

The point $z:=\frac{1}{1-\lambda}\sum_{i=N+1}^\infty \lambda_i x_i$ is clearly an element of $\overline{C}$. Since $x=\lambda y + (1-\lambda)z$, the result follows from Lemma~\ref{L:intbound}.
\end{proof}

\begin{lemma} \label{L:lemma3}
Let $X$ be a Banach space and let for each $j,k\in \en$, $f_k^j\in B_X$. Let $(a_k)_{k=1}^\infty \subset \ell_1$ and $(b_k)_{k=1}^\infty \subset \ell_1$ be two sequences of elements of $\ell_1$ such that $\lim_{k\to \infty} \|a_k-b_k\|_1=0$. 
Then 
$$\lim_{k\to \infty}\sum_{j=1}^\infty a_k^j f_k^j = \lim_{k\to \infty}\sum_{j=1}^\infty b_k^j f_k^j$$
whenever one of the limits exists.
\end{lemma}

\begin{proof}
Since $X$ is a Banach space, all the series in the statement, as well as the estimate below, converge because they converge absolutely. Hence, the following suffices to conclude the proof:
\begin{IEEEeqnarray*}{+rCl+x*}
& & \Big\| \sum_{j=1}^\infty a_k^j f_k^j   - \sum_{j=1}^\infty b_k^j f_k^j  \Big\| = \Big\| \sum_{j=1}^\infty \big(a_k^j-b_k^j\big)f_k^j \Big\| \\
 & \leq  & \sum_{j=1}^\infty \big|a_k^j-b_k^j \big| \cdot \big\|f_k^j\big\|\leq \sum_{j=1}^\infty \big|a_k^j - b_k^j \big|= \|a_k-b_k\|_1 \longrightarrow 0.
& \qedhere
\end{IEEEeqnarray*}
\end{proof}

The following lemma is the core result of the present section showing how to construct certain renormings of non-reflexive Banach spaces. The proof is relatively complicated from the technical point of view, but except for its use of the James's theorem \cite{Jam63} (stating that on a non-reflexive Banach space there always is a functional which does not attain its norm on the closed unit ball) it is elementary. The proof is divided into 20 paragraphs to make it easier to read and also to relate it to the proof of Proposition~\ref{P:ellonecase} which follows the same scheme. The main idea of the proof appears already in the second paragraph where we define the renorming by choosing a new unit ball for the space. The rest of the proof is to show that the image of this new unit ball via the quotient mapping is exactly what we want it to be (i.e. the set $A$). The difficult part is to show that the image contains no extra points, and that follows from the choice of renorming which uses the non-norm-attaining functional and the subsequent existence of a ``very non-convergent'' bounded sequence of points (a norming sequence for the functional contained in the unit sphere).

\begin{lemma} \label{T:Xgen}
Let $X$ be a non-reflexive Banach space, $d\in \en$, $Y\subset X$ be a closed subspace of codimension $d$ in $X$, and $q:X\to X/Y$ be the quotient mapping. Let $\|\cdot\|$ be an equivalent norm on $X/Y$ and let $A\subset X/Y$ be a symmetric convex $F_\sigma$-set satisfying $U_{(X/Y,\|\cdot\|)}\subset A \subset B_{(X/Y,\|\cdot\|)}$ for which there exists $v\in \partial A$ and $r\in(0,1)$ such that $A\setminus \big(U_{\|\cdot\|}(v,r)\cup U_{\|\cdot\|}(-v,r)\big)$ is closed.

Then $X$ admits an equivalent norm $\vertiii{\cdot}$ such that $q(B_{(X,\vertiii{\cdot})})=A$.
\end{lemma}

\begin{proof}

\pfpar
Since $Y$ is a closed subspace of finite codimension $d$ in $X$, it is complemented in $X$, its (topological) complement is isomorphic to $\er^d$ as well as to $X/Y$, and $X\cong Y\oplus_\infty X/Y$ (see e.g. \cite[Section 4.1]{FHHMZ}). It follows that $X\cong Y\oplus_\infty \er^d$. Note also that $Y$ is non-reflexive as in the opposite case $X$ would be reflexive, being isomorphic to the direct sum of two reflexive spaces, which is easily seen to be reflexive. 

Our goal is to find a certain equivalent norm on $X$; we can therefore assume that $X = Y\oplus_\infty \er^d$ and consider the projection to the second coordinate $q:X\to \er^d$ instead of the quotient mapping $X\to X/Y$. This is, indeed, correct as the topological and linear structure remains the same, so we can also assume without loss of generality that $A\subset \er^d$ and that $\|\cdot\|$ is a norm on $\er^d$ (and it is the one satisfying $B_{(\er^d, \|\cdot\|)}=\overline{A}$). To avoid any confusion, let us emphasize that we shall use no other norm on $\er^d$ in this proof. 

\pfpar
Now, denote $F_0:=\conv\big(A\setminus (U(v,r)\cup U(-v,r))\big)$; then $F_0$ is a compact subset of $A$. Clearly, the set $A\cap U(v,r)$ is of the type $F_\sigma$, so there are closed sets $E_j$, $j\geq 1$, such that $A\cap U(v,r)=\bigcup_{j=1}^\infty E_j$. From the convexity of $A\cap U(v,r)$ it follows that $A\cap U(v,r)=\bigcup_{j=1}^\infty F_j$ where $F_j:=\conv(E_j)$; then $F_j$ is also closed for each $j$.  
 
To define an equivalent norm with the desired properties, we need to use the fact that $Y$ is non-reflexive. The James's theorem implies that there exists $y^*\in Y^*\setminus \NA (Y)$, and we can fix a sequence $\big(z^j\big)_{j=0}^\infty \subset S_Y$ such that $y^*\big(z^j\big)\nearrow\|y^*\|$ as $j\to\infty$.
The following set $B$ clearly is the closed unit ball for an equivalent norm $\vertiii{\cdot}$ on $X$.
\begin{equation*}
B:=\cconv\,\Big(B_Y\times \frac{1}{2}A\cup \bigcup_{j=0}^\infty \big(\big\{z^j\big\}\times F_j \cup \big\{-z^j\big\}\times (-F_j)\big) \Big)
\end{equation*}
Indeed, $B$ is closed absolutely convex with $B_Y\times\frac{1}{2}A\subset B \subset B_Y\times \overline{A}$.

\pfpar
We will be done when we prove that $q(B)=A$. The inclusion $q(B)\supset A$ is easy to see; indeed, $A$ is symmetric so $A=\bigcup_{j=0}^\infty (F_j\cup(-F_j))$, and obviously $\bigcup_{j=0}	^\infty (F_j\cup(-F_j))\subset q(B)$. It remains to prove that $q(B)\subset A$.

Note that the sequence $(z^j)_{j=0}^\infty$ does not converge; our task would therefore be easy if the definition of $B$ only involved closure instead of the closed convex hull. If that were the case, our argument would be based on the simple observation that a sequence of the form $(z^{j_k},f_k)_{k=1}^\infty$ with $f_k\in F_{j_k}$ converges only if the sequence of indices $(j_k)_{k=1}^\infty$ is eventually constant with $\lim_{k\to\infty}j_k=j_\infty$, whence $\lim_{k\to\infty} f_k \in F_{j_\infty}\subset A$.

The rest of the proof deals with the technical difficulties arising from the presence of convex combinations.

\pfpar
To that end, let us fix an arbitrary point $u\in q(B)$ and an $x\in Y$ such that $(x,u)\in B$. Then $(x,u)$ can be written as the limit of a sequence of convex combinations:
\begin{equation}\label{E:limconv}
\begin{aligned}
(x,u)&= \lim_{k\to\infty} \Big( t_k\Big(x_k,\frac{1}{2}v_k\Big)+\sum_{j=0}^{n_k}\big(\alpha^j_k\big(z^j, f^j_k\big) - \beta^j_k\big(z^j, g^j_k\big)\big)    \Big)\\
&= \lim_{k\to\infty} \Big( t_k x_k + \sum_{j=0}^{n_k} \big(\alpha^j_k z^j - \beta^j_k z^j\big), \frac{1}{2}t_k v_k + \sum_{j=0}^{n_k}\big(\alpha^j_k f^j_k - \beta^j_k g^j_k\big) \Big)
\end{aligned}
\end{equation}
where $n_k\in \en$, $t_k+\sum_{j=0}^{n_k}\big(\alpha^j_k+\beta^j_k\big)=1$ with all the coefficients non-negative, $v_k\in A$ and $f^j_k, g^j_k \in F_j$ for each $k\in \en$ and $j\in\en\cup\{0\}$. Note that all the sets $F_j$, as well as $A$, are convex, and so in the convex combinations it is enough to take only one element from each of them instead of a convex combination of finitely many elements. 

In paragraphs \pfparcite{5} to \pfparcite{14} of this proof we essentially make a series of simplifying assumptions which will be shown to cause no loss of generality; in this manner we shall reduce the matter enough to be able to formulate the final argument in the last six paragraphs. 

\pfpar
First, let us formally make all the sums in \eqref{E:limconv} infinite series by defining for each $k\in \en$ and each $j>n_k$,
\begin{equation} \label{E:formallyinf}
\alpha_k^j:=0, \quad \beta_k^j:=0, \quad f_k^j:=f_k^1 \quad \text{and} \quad g_k^j:=g_k^1.
\end{equation} 
Thus we now have
\begin{equation}\label{E:seriesdef}
(x,u)=\lim_{k\to\infty} \Big( t_k x_k + \sum_{j=0}^\infty \big(\alpha^j_k z^j - \beta^j_k z^j\big), \frac{1}{2}t_k v_k + \sum_{j=0}^\infty\big(\alpha^j_k f^j_k - \beta^j_k g^j_k\big) \Big)
\end{equation}
where each of the infinite series has only finitely many non-zero summands (for each $k\in\en$). In the following we shall often use this fact without recalling it.

\pfpar
By passing to a subsequence (using a diagonal argument) we can assume that all the following limits exist for all $j\in\en$ (recall that all the sequences $(f^j_k)_{k=1}^\infty$ and $(g^j_k)_{k=1}^\infty$ are bounded in $\er^d$):
\begin{equation} \label{E:limits}
\begin{split}
\alpha^j := \lim_{k\to\infty} \alpha^j_k, \quad \beta^j:=\lim_{k\to\infty} \beta^j_k, \quad t:= \lim_{k\to\infty} t_k, \\
f^j:= \lim_{k\to\infty} f^j_k \quad \text{and} \quad g^j:= \lim_{k\to\infty} g^j_k.
\end{split}
\end{equation}

\pfpar
We now distinguish two cases: Either $t=0$ or $t>0$. Assume first $t>0$; then we shall conclude that $u\in A$ as follows. We can find $k_0\in \en$ such that $t_k>\frac{t}{2}$ for each $k\geq k_0$. Hence, for each $k\geq k_0$ we have
\begin{equation*}
\begin{split}
\Big\| \frac{1}{2}t_k v_k + \sum_{j=0}^{\infty}\big(\alpha^j_k f^j_k - \beta^j_k g^j_k\big) \Big\| \leq \Big\|\frac{t_k}{2} v_k\Big\| + \sum_{j=0}^\infty \big\| \alpha^j_kf^j_k \big\| + \sum_{j=0}^\infty \big\| \beta^j_k g^j_k \big\| \\
\leq \frac{t_k}{2} + \sum_{j=0}^\infty \big(\alpha^j_k + \beta^j_k \big) = t_k + \sum_{j=0}^\infty \big(\alpha^j_k + \beta^j_k \big) - \frac{t_k}{2} = 1 - \frac{t_k}{2} < 1-\frac{t}{4}.
\end{split}
\end{equation*}
Hence,
$$u=\lim_{k\to\infty} \Big(\frac{1}{2}t_k v_k + \sum_{j=0}^{\infty}\big(\alpha^j_k f^j_k - \beta^j_k g^j_k\big)\Big) \in \Big(1-\frac{t}{4}\Big) B_{(\er^d,\|\cdot\|)}\subset U_{(\er^d,\|\cdot\|)} \subset A.$$

\pfpar
Till the end of the proof, we shall therefore investigate the case when $t=\lim_{k\to\infty}t_k=0$. It follows from \eqref{E:seriesdef} that 
$$u=\lim_{k\to\infty} (1-t_k)\sum_{j=1}^\infty \Big ( \frac{\alpha^j_k}{1-t_k} f^j_k - \frac{\beta^j_k}{1-t_k} g^j_k \Big ) = \lim_{k\to \infty}\sum_{j=1}^\infty \Big ( \frac{\alpha^j_k}{1-t_k} f^j_k - \frac{\beta^j_k}{1-t_k} g^j_k \Big ),$$
and similarly for the first coordinate $x$; the sum on the right is a convex combination for each $k\in \en$.
Hence, we can assume without loss of generality that $t_k=0$ for each $k$, so we have
\begin{IEEEeqnarray*}{c}
\sum_{j=0}^\infty\big(\alpha^j_k + \beta^j_k\big) = 1 \quad\text{for each $k$, and}\\
(x,u)=\lim_{k\to\infty} \Big(\sum_{j=0}^\infty\big(\alpha^j_k z^j - \beta^j_k z^j\big), \sum_{j=0}^\infty\big(\alpha^j_k f^j_k - \beta^j_k g^j_k\big)\Big).
\end{IEEEeqnarray*}

\pfpar 
Next assumption we want to make is $\alpha^0_k=\beta^0_k=0 $ for all $k\in \en$. In the present paragraph we show that this causes no loss of generality:

Let us first consider the case when $\alpha^0+\beta^0=1$; then 
$$\lim_{k\to\infty}\sum_{j=1}^\infty \big( \alpha^j_k + \beta^j_k \big)=\lim_{k\to\infty} \big(1-\big(\alpha^0_k+\beta^0_k\big)\big)=1-\big(\alpha^0+\beta^0\big)=0.$$
It is now easy to see (e.g. using Lemma~\ref{L:lemma3}) that $\lim_{k\to \infty} \sum_{j=1}^\infty \big(\alpha^j_k f^j_k - \beta^j_k g^j_k \big) = 0$, whence (by \eqref{E:seriesdef} and \eqref{E:limits}) $u=\alpha^0 f^0 + \beta^0 g^0\in F_0\subset A$.

Assume now that $\alpha^0+\beta^0<1$. Clearly we can also assume that for each $k\in \en$ we have $\alpha^0_k+ \beta^0_k < 1$, and consider the limit of convex combinations
\begin{equation}\label{E:limofconvexcomb}
 \tilde{u}:=\lim_{k\to \infty} \sum_{j=1}^\infty \Big( \frac{\alpha^j_k}{1-\big(\alpha^0_k+\beta^0_k\big)}f^j_k - \frac{\beta^j_k}{1-\big(\alpha^0_k+\beta^0_k\big)}g^j_k  \Big);
\end{equation}
it is easy to see that the limit exists. Then we have
\begin{IEEEeqnarray*}{rCl}
u & = & \big(\alpha^0+\beta^0\big) \Big( \frac{\alpha^0}{\alpha^0+\beta^0}f^0 - \frac{\beta^0}{\alpha^0+\beta^0}g^0\Big) \\ 
\yesnumber\label{E:udecomp}  && + \> \big(1-\big(\alpha^0+\beta^0\big)\big)\lim_{k\to\infty} \sum_{j=1}^\infty \Big( \frac{\alpha^j_k}{1-\big(\alpha^0+\beta^0\big)}f^j_k - \frac{\beta^j_k}{1-\big(\alpha^0+\beta^0\big)}g^j_k  \Big) \\
	& = & \big(\alpha^0+\beta^0\big) \Big( \frac{\alpha^0}{\alpha^0+\beta^0}f^0 - \frac{\beta^0}{\alpha^0+\beta^0}g^0\Big) + \big(1-\big(\alpha^0+\beta^0\big)\big) \tilde{u},
\end{IEEEeqnarray*}

where the last equality follows from Lemma~\ref{L:lemma3} whose assumption is verified as follows:
\begin{IEEEeqnarray*}{rCl}
& & \Big\| \Big(\frac{\alpha^j_k}{1-\big(\alpha^0+\beta^0 \big)} \Big)_{j=1}^\infty - \Big(\frac{\alpha^j_k}{1-\big(\alpha^0_k+\beta^0_k \big)} \Big)_{j=1}^\infty \Big\|_1 \\ 
& = & \Big| \frac{1}{1-\big(\alpha^0 + \beta^0 \big)}-\frac{1}{1-\big(\alpha^0_k + \beta^0_k \big)} \Big|\cdot \big\| \big(\alpha^j_k \big)_{j=1}^\infty \big\|_1 \longrightarrow 0, \quad k\to\infty,
\end{IEEEeqnarray*}
and similarly with $\beta^j_k$ instead of $\alpha^j_k$. 

In \eqref{E:udecomp} we expressed $u$ as a convex combination of an element of $F_0\subset A$ and $\tilde{u}$; therefore, to see that $u\in A$, it is sufficient to prove $\tilde{u}\in A$. Of course, $\tilde{u}$ is in $q(B)$ as it is the projection of $(\tilde{x}, \tilde{u})\in B$ where 
\begin{equation*}
\tilde{x} :=\lim_{k\to \infty} \sum_{j=1}^\infty \Big( \frac{\alpha^j_k}{1-\big(\alpha^0_k+\beta^0_k\big)}z^j - \frac{\beta^j_k}{1-\big(\alpha^0_k+\beta^0_k\big)}z^j  \Big).
\end{equation*}

To avoid introducing more notation, we henceforth assume that $\alpha^0_k=\beta^0_k=0 $ for all $k\in \en$, so we now have
\begin{equation*}
(x,u)=\lim_{k\to\infty} \Big( \sum_{j=1}^\infty \big( \alpha^j_k z^j- \beta^j_k z^j \big) , \sum_{j=1}^\infty \big( \alpha^j_k f^j_k - \beta^j_k g^j_k \big) \Big).
\end{equation*}

\pfpar
Next we distinguish two subcases depending on the value of 
$$\nu := \min\Big\{ \lim_{k\to\infty} \sum_{j=1}^\infty \alpha^j_k\,, \:\lim_{k\to\infty}\sum_{j=1}^\infty \beta^j_k\Big\};$$
by passing to a subsequence we can guarantee that both limits exist.

\pfpar
Let us first consider the case when $\nu >0$; we claim that then $u\in A$. To show that, set for $k\in \en$,  
$$\alpha_k:=\sum_{j=1}^\infty\alpha^j_k, \quad \beta_k:=\sum_{j=1}^\infty\beta^j_k, \quad f_k:=\frac{1}{\alpha_k}\sum_{j=1}^\infty\alpha^j_k f^j_k\quad\text{and} \quad g_k:=\frac{1}{\beta_k}\sum_{j=1}^\infty\beta^j_k g^j_k,$$
and find $k_0\in\en$ such that for each $k\geq k_0$, $\alpha_k, \beta_k\in \big(\frac{\nu}{2},1-\frac{\nu}{2} \big)$. Now, for a given $k\geq k_0$, let us assume for example $\alpha_k\geq \beta_k$; the opposite case is analogous. 

Recall that for $j\geq 1$,  $f^j_k, g^j_k\in F_j\subset U(v,r)$ which implies that $f_k, g_k\in U(v,r)$, and thus
$\|f_k-g_k\|< 2r$. We now have
\begin{IEEEeqnarray*}{rCl}
 & & \Big\| \sum_{j=1}^\infty\big(\alpha_k^j f_k^j - \beta_k^j g_k^j\big) \Big\| = \|\alpha_k f_k - \beta_k g_k \| = \|(\alpha_k-\beta_k)f_k + \beta_k (f_k-g_k)\| \\
 & \leq & \alpha_k-\beta_k + 2r\beta_k = 1- 2\beta_k + 2r\beta_k = 1-2\beta_k(1-r)<1-\nu(1-r)<1.
\end{IEEEeqnarray*}
This is true for any $k\geq k_0$, and consequently
$$u=\lim_{k\to\infty}\sum_{j=1}^\infty\big(\alpha_k^j f_k^j - \beta_k^j g_k^j\big) \in (1-\nu(1-r))\cdot B_{(\er^d,\|\cdot\|)} \subset U_{(\er^d,\|\cdot\|)} \subset A.$$

\pfpar
We now turn to the case $\nu=0$. The first step is to observe that we can assume without loss of generality that either $\alpha_k=0$ for all $k$, or $\beta_k=0$ for all $k$. Setting 
$$c_k:=\big(\alpha_k^1,\alpha_k^2,\dots\big)\quad \text{and}\quad d_k:=\big(\beta_k^1, \beta_k^2, \dots\big),$$
we see that either $\|c_k\|_1=\alpha_k\to 0$ or $\|d_k\|_1=\beta_k\to 0$. Let us examine only the second case as the first one is analogous. Then we have (e.g. by Lemma~\ref{L:lemma3}) 
$$\lim_{k\to\infty}\sum_{j=1}^\infty \beta_k^j g_k^j=0,$$
and therefore (recall that $\alpha_k+\beta_k=1$, whence $\alpha_k\to 1$)
$$\lim_{k\to\infty} \sum_{j=1}^\infty \frac{\alpha_k^j}{\alpha_k} f_k^j = \lim_{k\to\infty}\sum_{j=1}^\infty \alpha_k^j f_k^j =
\lim_{k\to\infty}\sum_{j=1}^\infty \big(\alpha_k^j f_k^j-\beta_k^j g_k^j\big)=u,$$
from where the observation follows. 

From now on, we assume that $\beta_k^j=0$ for all $j,k\in\en$, so 
\begin{equation}\label{E:xusimple}
(x,u)=\lim_{k\to\infty} \Big( \sum_{j=1}^\infty \alpha_k^j z^j\,,\: \sum_{j=1}^\infty \alpha_k^j f_k^j \Big).
\end{equation}

\pfpar
Again we need to distinguish two further subcases: This time depending on the value of 
\begin{equation*}
\lambda := \sum_{j=1}^\infty \alpha^j.
\end{equation*}

\pfpar
It is easy to see that $\lambda \in [0,1]$; the easier situation is when $\lambda=1$: In this case we need to observe that the limits in the definition of $u$ commute, that is:
\begin{equation}\label{E:commute}
u=\lim_{k\to\infty}\sum_{j=1}^\infty \alpha^j_k f^j_k =\sum_{j=1}^\infty \alpha^j f^j.
\end{equation}
Indeed, let $\varepsilon>0$ be given and let us find $j_0\in \en $ such that $\sum_{j=j_0+1}^\infty \alpha^j <\frac{\varepsilon}{4}$, and $k_0\in \en$ such that for each $k\geq k_0$ we have $\big\|\sum_{j=1}^{j_0}\big(\alpha^j_k f^j_k -\alpha^j f^j\big)\big\| < \frac{\varepsilon}{4}$ and $\big| \sum_{j=1}^{j_0}\big(\alpha^j_k - \alpha^j \big)\big|<\frac{\varepsilon}{4}$. Then for each $k\geq k_0$ we also obtain (using the assumption $\lambda=1$) that $\big| \sum_{j=j_0+1}^\infty \big(\alpha^j_k - \alpha^j\big)  \big|<\frac{\varepsilon}{4}$, and so
\begin{IEEEeqnarray*}{rCl}
\IEEEeqnarraymulticol{3}{l}{\Big\|\sum_{j=1}^\infty \alpha^j_k f^j_k - \sum_{j=1}^\infty \alpha^j f^j  \Big\|}\\
\qquad &\leq &  \Big\|\sum_{j=1}^{j_0}\big(\alpha^j_k f^j_k -\alpha^j f^j \big) \Big\|  + \Big\| \sum_{j=j_0+1}^\infty\alpha^j_k f^j_k\Big\| + \Big\| \sum_{j=j_0+1}^\infty \alpha^j f^j \Big\|   \\
&<   &\frac{\varepsilon}{4} + \sum_{j=j_0+1}^\infty \alpha^j_k + \sum_{j=j_0+1}^\infty \alpha^j < \frac{\varepsilon}{4} + \Big(\sum_{j=j_0+1}^\infty \alpha^j + \frac{\varepsilon}{4}\Big) + \frac{\varepsilon}{4}  <  \,\varepsilon.
\end{IEEEeqnarray*}
This proves the second equality in \eqref{E:commute}, which implies that $u\in \sconv(A)$ (note that $f^j, g^j\in F_j\subset A$ for every $j\in \en$) and it follows by Proposition~\ref{P:sconvisconv} that $u\in \conv(A)=A$.

\pfpar
Finally, assume that $\lambda\in [0,1)$; we will show that this assumption leads to a contradiction. Set
\begin{equation*}
x_\lambda:=\sum_{j=1}^\infty \alpha^j z^j \quad \text{and} \quad x_{(1-\lambda)}:=x-x_\lambda = \lim_{k\to\infty} \sum_{j=1}^\infty \big(\alpha_k^j-\alpha^j\big)z^j.
\end{equation*}
Now we aim to prove that $\frac{x_{(1-\lambda)}}{1-\lambda}$ is the limit of a sequence of convex combinations of $z^j$. But it need not suffice to simply divide all the coefficients $\alpha_k^j-\alpha^j$ by $1-\lambda$ as some of them could be negative. Hence, for any $a\in\er$ we define
$$a^+:=\max\{a, 0\},\quad  a^-:=(-a)^+, \quad\text{and}\quad \tilde{\gamma}_k^j:=\big(\alpha_k^j-\alpha^j\big)^+,$$
and we claim that 
$$x_{(1-\lambda)}=\lim_{k\to\infty} \sum_{j=1}^\infty \tilde{\gamma}_k^j z^j.$$

\pfpar
By Lemma~\ref{L:lemma3} it is enough to prove that 
$$\lim_{k\to\infty} \big\|\big(\alpha_k^j-\alpha^j\big)_{j=1}^\infty - \big(\tilde{\gamma}_k^j\big)_{j=1}^\infty\big\|_1=0.$$
Let an arbitrary $\varepsilon>0$ be given. We can find $N\in\en$  such that $\sum_{j=N+1}^\infty \alpha^j < \frac{\varepsilon}{2}$ and $k_0\in\en$ such that for each natural $k\geq k_0$ we have $\sum_{j=1}^N \big|\alpha_k^j-\alpha^j\big|<\frac{\varepsilon}{2}$. Note that $\big( \alpha_k^j - \alpha^j \big)^- \leq \alpha^j$ as $\alpha_k^j\geq 0$. Then for each $k\geq k_0$ we have
$$\big\|\big(\alpha_k^j-\alpha^j\big)_{j=1}^\infty - \big(\tilde{\gamma}_k^j\big)^\infty_{j=1}\big\|_1=\sum_{j=1}^\infty \big(\alpha_k^j-\alpha^j\big)^- \leq \sum_{j=1}^N \big| \alpha_k^j -\alpha^j \big|+\sum_{j=N+1}^\infty \alpha^j <\varepsilon.$$

\pfpar
Now, set $\sigma_k:=\sum_{j=1}^\infty \tilde{\gamma}_k^j$; then the last part together with the fact that $\sum_{j=1}^\infty \big(\alpha_k^j-\alpha^j \big)=1-\lambda$ also implies that $\lim_{k\to\infty}\sigma_k=1-\lambda$. Thus, setting $\gamma_k^j:=\frac{\tilde{\gamma}_k^j}{\sigma_k}$ we obtain
\begin{equation*}
x_{(1-\lambda)}=\lim_{k\to\infty} \frac{1-\lambda}{\sigma_k} \sum_{j=1}^\infty \tilde{\gamma}_k^j z^j = (1-\lambda) \lim_{k\to\infty} \sum_{j=1}^\infty \gamma_k^j z^j,
\end{equation*}
and from the definition it immediately follows that
$$\sum_{j=1}^\infty \gamma_k^j=1 \text{ for each $k$, and }\lim_{k\to\infty}\gamma_k^j=0 \text{ for each $j$.}$$
\pfpar
From these last facts we can readily see that 
\begin{equation}\label{E:intersect}
\frac{x_{(1-\lambda)}}{1-\lambda}=\lim_{k\to\infty}\sum_{j=1}^\infty \gamma_k^j z^j \in \bigcap_{N=1}^\infty \cconv\big\{z^j\setsep j > N\big\}.
\end{equation}
Indeed, pick arbitrary $N\in \en$ and $\varepsilon>0$, and find $k_0$ such that for each natural $k\geq k_0$,
$$ \gamma_k:=\sum_{j=1}^N \gamma_k^j  < \frac{\varepsilon}{2}. $$ 
Then for each $k\geq k_0$,
\begin{IEEEeqnarray*}{rCl}
\IEEEeqnarraymulticol{3}{l}{\dist\Big(\sum_{j=1}^\infty \gamma_k^j z^j, \cconv \big\{z^j\setsep j>N \big\} \Big)} \\
\qquad & \leq & \Big\| \sum_{j=1}^\infty \gamma_k^j z^j - \frac{1}{1-\gamma_k}\sum_{j=N+1}^\infty \gamma_k^j z^j \Big\|   \\
& \leq & \Big\| \sum_{j=1}^N \gamma_k^j z^j \Big\| + \Big\| \Big(1-\frac{1}{1-\gamma_k} \Big) \sum_{j=N+1}^\infty \gamma_k^j z^j \Big\| \\
& \leq & \gamma_k + \Big| 1-\frac{1}{1-\gamma_k} \Big|(1-\gamma_k)=\gamma_k+\gamma_k  < \varepsilon.
\end{IEEEeqnarray*}
This estimate yields that $\frac{x_{(1-\lambda)}}{1-\lambda} \in \cconv \big\{ z^j \setsep j>N \big\}$.

\pfpar
But this is impossible since the intersection in \eqref{E:intersect} is empty as we will show below. On the other hand, the limit in the definition \eqref{E:xusimple} of $x$ converges by the assumption and since $x_\lambda$ is a well-defined element of the Banach space $Y$, the limit defining $x_{(1-\lambda)}$ necessarily converges as well. This is a contradiction showing that the case $\lambda\in[0,1)$ does not occur; since in all the other cases we have already shown that $u\in A$, the proof will be complete, once we make the observation that
$$ C:=\bigcap_{N=1}^\infty \cconv\big\{z^j\setsep j > N\big\} = \emptyset.$$

\pfpar
Recall that $\big (z^j \big)_{j=1}^\infty \subset S_Y$ was chosen to be a norming sequence of the non-norm-attaining functional $y^*$. If there existed a point $z\in C$, then $z\in B_Y$ and $y^*(z)>y^*\big(z^N\big)$ for each $N\in\en$. Hence, $y^*(z)=\|y^*\|$ which is a contradiction, and the proof is complete.
\end{proof}

\noindent As a corollary we obtain the following result.

\begin{theorem}\label{C:counterex}
Let $X$ be a non-reflexive Banach space and $Y\subset X$ a closed subspace of finite codimension $d\geq 2$. Then $X$ admits an equivalent norm such that $Y^\bot \subset\NA(X)$ and $Y$ is not proximinal.
\end{theorem}

\begin{proof}
As explained in \pfparcite{1} of the proof of Lemma~\ref{T:Xgen}, we can assume that $X/Y$ is isomorphic to $\er^d$ where $d$ is the codimension of $Y$ in $X$. Let $\|\cdot\|_2$ be the Euclidean norm on $\er^d$ and let $e_i$, $i=1,\dots,d$, be the canonical basis in $\er^d$. Put
$$A:=\conv\big(B_{\|\cdot\|_2}(e_1,1)\cup B_{\|\cdot\|_2}(-e_1,1)\big)\setminus \{e_1+e_2, -e_1-e_2\}.$$
Clearly $A$ is a symmetric convex $F_\sigma$ set with nonempty interior satisfying the condition from Lemma~\ref{T:Xgen}, so there is a renorming $\vertiii{\cdot}$ of $X$ such that $q(B_X)=A$ where $q:X\to \er^d$ is the quotient mapping; then $\overline{A}$ is the closed unit ball in $\er^d\cong X/Y$ with the corresponding norm. In the sequel, by $X$ we mean $(X,\vertiii{\cdot})$. 

Since $q(B_X)=A\neq B_{X/Y}$, Lemma~\ref{L:proxchar} yields that $Y$ is not proximinal in $X$. 

To prove that $Y^\bot\subset \NA(X)$ we will use Lemma~\ref{L:proxchar} again. Take any $\varphi\in (X/Y)^*$; we are to show that it attains its norm at a point of $q(B_X)$. Since $\dim(X/Y)<\infty$, there is a point $u\in\att(\varphi)$. If $u\notin\{e_1+e_2, -e_1-e_2\}$, we are done. Suppose $u= e_1+e_2$; of course, the case $u=-e_1-e_2$ is symmetrical. Further, we can clearly assume $\|\varphi\|=1$, so $\varphi(u)=1$. 

Consider the hyperplane $H:=\{(x_1,\dots,x_n)\in\er^d\setsep x_2=1\}$; obviously $H$ is a tangent hyperplane to $\overline{A}$ at the point $e_1+e_2$ (i.e. $H\cap \Int(\overline{A})=\emptyset$ and $e_1+e_2\in H\cap \overline{A}$). But there is only one such hyperplane because $\overline{A}\supset B_{\|\cdot\|_2}(e_1,1)$ and $e_1+e_2$ lies on the boundary of $B_{\|\cdot\|_2}(e_1,1)$. It follows that $H=\varphi^{-1}(1)$. But $e_2\in H\cap q(B_X)$ which concludes the proof.
\end{proof}

\begin{proposition}\label{P:ellonecase}
Let $Y\subset\ell_1$ be a closed subspace of codimension $d\in \en$, and $q:\ell_1\to \ell_1/Y$ be the quotient mapping. Let $\|\cdot\|$ be an equivalent norm on $\ell_1/Y$ and let $A\subset \ell_1/Y$ be a symmetric convex $F_\sigma$-set satisfying $U_{(\ell_1/Y,\|\cdot\|)}\subset A\subset B_{(\ell_1/Y,\|\cdot\|)}$. Then there is an equivalent norm $\vertiii{\cdot}$ on $\ell_1$ such that $q(B_{(\ell_1, \vertiii{\cdot})})=A$.
\end{proposition}

\setcounter{pfparcount}{0}
\begin{proof}
The proof is very similar to that of Lemma~\ref{T:Xgen}, and so we will do it more briefly, emphasizing the necessary changes in each paragraph of the proof.

\pfpar 
Denote $X:=\ell_1$. The subspace $Y\subset X$ is isomorphic to $Z:= \big \{\big(x^j \big)_{j=1}^\infty \in\ell_1 \setsep  x^1=x^2=\dots=x^d = 0 \big\}$ as all subspaces of the same finite codimension in $X$ are isomorphic. But $Z$ is obviously (isometrically) isomorphic to $\ell_1$, and so we can (similarly as in the proof of Lemma~\ref{T:Xgen}) assume that $X=Y\oplus_\infty \er^d $ where $Y=\ell_1$. Throughout the proof we shall only consider the norm $\|\cdot\|$ on $\er^d$ for which $B_{(\er^d,\|\cdot\|)} = \overline{A}$.

\pfpar 
In a similar (but simpler) way as in \ref{T:Xgen} we can find closed convex sets $F_j$ ($j\in\en$) of diameter less than $1$ and such that $A=\bigcup_{j=1}^\infty F_j$. 

Now, instead of using the James's theorem, we take the canonical basis $(e^j)_{j=1}^\infty$ of $Y=\ell_1$ and we define
$$B:=\cconv\,\Big(B_Y\times \frac{1}{2}A\cup \bigcup_{j=1}^\infty \big(\big\{e^j\big\}\times F_j \cup \big\{-e^j\big\}\times (-F_j)\big) \Big).$$ 
Again, $B$ is the closed unit ball for an equivalent norm $\vertiii{\cdot}$ on $X$. 

\pfpar 
We want to prove that $q(B)=A$; the inclusion $q(B)\supset A$ follows immediately from the definition of $B$ and the fact that $A=\bigcup_{j=1}^\infty F_j$. 

\pfpar
To prove the converse inclusion, we again take arbitrary $u\in q(B)$ and $x\in Y$ such that $(x,u)\in B$. As in \ref{T:Xgen},
\begin{equation}\label{E:xudefellcase}
(x,u) = \lim_{k\to\infty} \Big( t_k x_k + \sum_{j=1}^{n_k} \big(\alpha^j_k e^j - \beta^j_k e^j\big), \frac{1}{2}t_k v_k + \sum_{j=1}^{n_k}\big(\alpha^j_k f^j_k - \beta^j_k g^j_k\big) \Big)
\end{equation}
where $n_k\in\en$, $t_k+\sum_{j=1}^{n_k}\big(\alpha^j_k+\beta^j_k\big)=1$ with $\alpha_k^j, \beta_k^j\geq 0$, $v_k\in A$ and $f_k^j, g_k^j\in F_j$.

\pfparref, \pfpar
Setting $\alpha^j_k$, $\beta^j_k$, $f^j_k$ and $g^j_k$ for $k\in \en$ and $j>n_k$ as in \eqref{E:formallyinf}, we can assume again that all the limits in \eqref{E:limits} exist. As in \ref{T:Xgen}, we now formally consider all the sums in \eqref{E:xudefellcase} as infinite series.

\pfparref, \pfpar
In the same way as in \ref{T:Xgen} we prove that we can assume all $t_k$'s to be zero (otherwise $u\in\Int(A)$). 

\pfpar This step of the proof we skip entirely because in this case we do not work with the set $F_0$.

\pfpar
Here we make one of the important changes by defining
$$\mu := \lim_{k\to\infty} \sum_{j=1}^\infty 	\min \big\{ \alpha_k^j, \beta_k^j \big\};$$
of course, we can assume that this limit exists, and we distinguish two cases: $\mu>0$ and $\mu=0$. 

\pfpar 
Suppose $\mu>0$. For $k\in \en$ set 
\begin{equation*}
u_k:=\sum_{j=1}^\infty \big(\alpha_k^j f_k^j - \beta_k^j g_k^j \big); 
\end{equation*}
then $u=\lim_{k\to\infty}u_k$. Further, for all $j,k\in \en$ set $\mu_k^j:=\min\big\{\alpha_k^j, \beta_k^j \big\}$. Fix $k\in \en$ and for a given $j\in\en$ assume that $\alpha_k^j\geq \beta_k^j$; the opposite case is analogous and we obtain the same estimate:
\begin{IEEEeqnarray*}{rCl}
\big\| \alpha_k^j f_k^j - \beta_k^j g_k^j \big\| &=& \big \| \big (\alpha_k^j-\mu_k^j \big)f_k^j + \mu_k^j \big( f_k^j-g_k^j \big) \big\| \\
&\leq& \big(\alpha_k^j - \mu_k^j\big) + \mu_k^j \\
&=& \alpha_k^j + \beta_k^j - \mu_k^j;
\end{IEEEeqnarray*}
here the inequality follows from the fact that $f_k^j, g_k^j\in F_j\subset B_{\er^d}$ and $\diam(F_j)<1$. Now, by taking the sum over all $j\in\en$ we obtain
$$\|u_k\| \leq \sum_{j=1}^\infty \big\| \alpha_k^j f_k^j - \beta_k^j g_k^j \big\| \leq \sum_{j=1}^\infty \big(\alpha_k^j + \beta_k^j - \mu_k^j \big) = 1 - \sum_{j=1}^\infty \mu_k^j.$$
It follows after passage to limit as $k\to\infty$ that $\|u\|\leq 1-\mu$, and so $u\in (1-\mu)B_{\er^d} \subset U_{\er^d} \subset A$.

\pfpar
For the rest of the proof we shall assume that $\mu=0$. Similarly as in the proof of Lemma~\ref{T:Xgen} we prove that there is no loss of generality in assuming that $\mu_k=0$ for each $k$. In other words, we can assume that for all $j,k\in \en$ we have $\alpha_k^j=0$ or $\beta_k^j=0$. 

\pfpar
Recalling the notation in \eqref{E:formallyinf}, we finally distinguish two cases depending on the value of 
$$ \lambda:= \sum_{j=1}^\infty \big( \alpha^j + \beta^j \big).$$

\pfpar
Again, $\lambda\in [0,1]$ and we first resolve the case $\lambda =1$: We shall observe that the limit and the series in the definition of $u$ commute, that is:
\begin{equation}\label{E:ell03}
u=\lim_{k\to\infty}\sum_{j=1}^\infty \big(\alpha_k^j f_k^j - \beta_k^j g_k^j\big)=\sum_{j=1}^\infty \big(\alpha^j f^j - \beta^j g^j \big).
\end{equation}
Indeed, let $\varepsilon>0$ be given and let us find $j_0\in \en $ such that $\sum_{j=j_0+1}^\infty \big(\alpha^j + \beta^j \big)<\frac{\varepsilon}{4}$, and $k_0\in \en$ such that for each $k\geq k_0$ we have $\big\|\sum_{j=1}^{j_0}\big(\alpha_k^j f_k^j -\alpha^j f^j - \big(\beta_k^j g_k^j - \beta^j g^j\big)\big)\big\| < \frac{\varepsilon}{4}$ and $\big| \sum_{j=1}^{j_0}\big(\alpha_k^j + \beta_k^j \big) - \sum_{j=1}^{j_0}\big(\alpha^j + \beta^j \big) \big|<\frac{\varepsilon}{4}$. By the same calculation as in the proof of \ref{T:Xgen} we obtain that for each $k\geq k_0$,
$$ \Big\|\sum_{j=1}^\infty \big(\alpha_k^j f_k^j - \beta_k^j g_k^j \big)- \sum_{j=1}^\infty \big(\alpha^j f^j - \beta^j g^j\big) \Big\| < \varepsilon. $$
This proves \eqref{E:ell03}, so $u\in \sconv(A)=\conv(A)=A$ by Proposition~\ref{P:sconvisconv} and the convexity of $A$. 

\pfpar
Aiming for a contradiction, suppose that $\lambda\in [0,1)$ and set
\begin{equation*}
\begin{aligned}
x_\lambda &:=\sum_{j=1}^\infty \big(\alpha^j - \beta^j \big) e^j \quad \text{and} \\
x_{(1-\lambda)} &:= x- x_\lambda = \lim_{k\to\infty} \sum_{j=1}^\infty \big(\alpha_k^j-\alpha^j - \big(\beta_k^j - \beta^j \big) \big) e^j. 	
\end{aligned}
\end{equation*}
Note that the last limit must exist as it is the difference of an existing limit and $x_\lambda$. Define
$$ \tilde{\gamma}_k^j:= \big( \alpha_k^j - \alpha^j \big)^+ \quad\text{and}\quad \tilde{\delta}_k^j:= \big( \beta_k^j - \beta^j \big)^+,$$
and we now claim that 
$$ x_{(1-\lambda)}=\lim_{k\to\infty} \sum_{j=1}^\infty \big(\tilde{\gamma}_k^j - \tilde{\delta}_k^j \big)e^j.$$

\pfpar
By Lemma~\ref{L:lemma3} it suffices to prove that 
$$\lim_{k\to\infty}\big\| \big(\alpha_k^j-\alpha^j\big)_{j=1}^\infty - \big( \tilde{\gamma}_k^j \big)_{j=1}^\infty \big\|_1 = 0 \quad\text{and}\quad
\lim_{k\to\infty}\big\| \big(\beta_k^j - \beta^j \big)_{j=1}^\infty - \big(\tilde{\delta}_k^j \big)_{j=1}^\infty\big\|_1 = 0.$$
This can be done in the same way as in the proof of \ref{T:Xgen}.

\pfpar
Now, set $\sigma_k:=\sum_{j=1}^\infty \big( \tilde{\gamma}_k^j + \tilde{\delta}_k^j\big)$	 and observe that $\lim_{k\to\infty}\sigma_k = 1-\lambda$. Further, define
$$\gamma_k^j:= \frac{\tilde{\gamma}_k^j}{\sigma_k}\quad\text{and}\quad\delta_k^j:= \frac{\tilde{\delta}_k^j}{\sigma_k};$$
then we have:
\begin{enumerate}[(i)]
\item $x_{(1-\lambda)} =(1-\lambda)\lim_{k\to\infty} \sum_{j=1}^\infty\big(\gamma_k^j - \delta_k^j \big)e^j$;
\item for each $k\in\en$, $\sum_{j=1}^\infty \big( \gamma^j_k + \delta^j_k \big) = 1$;
\item for all $j,k\in\en$, either $\gamma_k^j=0$ or $\delta_k^j=0$;
\item and for each $j\in\en$, $\lim_{k\to\infty} \gamma_k^j=0 \text{ and } \lim_{k\to\infty} \delta_k^j=0$.
\end{enumerate}

\pfparref, \pfparref, \pfpar
Now, on one hand we can readily see from (i), (ii) and (iii) that
$$\|x_{(1-\lambda)}\| = (1-\lambda)\lim_{k\to\infty}\sum_{j=1}^\infty \big(\gamma_k^j + \delta_k^j \big) = (1-\lambda)\lim_{k\to \infty}1 = 1-\lambda > 0.$$
On the other hand, (i) and (iv) yield that $x_{(1-\lambda)}=0$. Indeed, (i) says that $x_{(1-\lambda)}$ is the norm-limit of a certain sequence which by (iv) converges pointwise to $0$. 
\end{proof}

\begin{remark} 
\begin{enumerate}[(a)]
\item Let $X$ be a Banach space, $Y\subset X$ be a closed subspace and $q:X\to X/Y$ be the quotient mapping.
Since any functional $y^*\in Y^\bot$ can be expressed as $y^*=\varphi\circ q$ for some $\varphi:=(X/Y)^*$, one can notice that in the proof of Lemma~\ref{L:proxchar} (ii) we actually show the following equivalence for each $y^*\in Y^\bot$:
$$ y^*\in \NA(X) \Longleftrightarrow \text{the corresponding }\varphi\in(X/Y)^* \text{ attains its norm on } q(B_X). $$

Consider the situation when $X=\ell_1$ and $Y$ has codimension $d$ in $X$; then $X/Y\cong \er^d$. The set $A:=U_{(\er^d,\|\cdot\|_2)}$ satisfies the assumptions of Proposition~\ref{P:ellonecase}, so there is an equivalent norm $\vertiii{\cdot}$ on $\ell_1$ such that $q(B_{(\ell_1,\vertiii{\cdot})})=A$. From the above equivalence we immediately obtain that $Y^\bot\subset (X^*\setminus \NA(\ell_1,\vertiii{\cdot}))\cup\{0\}$. 

In other words, given any finite-dimensional subspace $Z\subset (\ell_1)^*$, we can find a renorming such that $Z\setminus\{0\}$ is contained in non-norm-attaining functionals.  Note that a result of F.J.~Garc\'ia-Pacheco \cite[Theorem 1.2]{Gar08} implies that $\ell_1^*\setminus\NA(\ell_1)$ is, in fact, spaceable.

It is also easy to see that, in the new norm, $Y$ is \emph{antiproximinal} (i.e. no point $x\notin Y$ has in $Y$ a closest point). Indeed, the proof of Lemma~\ref{L:proxchar}~\emph{(i)} also shows that 
$$ X \text{ is antiproximinal } \Longleftrightarrow q(B_X)=U_{X/Y}.$$

Likewise, given a closed subspace $Y$ of finite codimension in a non-reflexive Banach space $X$, there exists an equivalent norm on $X$ such that $Y$ is proximinal in $X$. This follows from Lemma~\ref{T:Xgen} as the set $A:=B_{X/Y}$ satisfies the assumptions.
\item The assumption that $Y$ is finite-codimensional is very important for our approach to both Lemma~\ref{T:Xgen} and Proposition~\ref{P:ellonecase}. In the proof of the former, we use this assumption to see that $Y$ is complemented, and therefore non-reflexive as its complement is finite-dimensional; this argument is not necessary in the latter case as $\ell_1$ contains no infinite-dimensional complemented reflexive subspaces---see \cite[Section 13.7]{FHHMZ}. The fact that $Y$ is non-reflexive is needed as any reflexive space is proximinal in any superspace (see \cite{Singer}; the converse is also true---\cite{Pollul} or \cite{Sin73}). In both cases we rely on Proposition~\ref{P:sconvisconv}. 
\item It is natural to ask whether the assumption that $A$ is $F_\sigma$ in Lemma~\ref{T:Xgen} and Proposition~\ref{P:ellonecase} can be relaxed to higher Borel classes or even analytic sets. Related questions have already been studied: R.~Kaufman proves in \cite{Kau91} that any non-reflexive Banach space admits an equivalent norm such that the set of norm-attaining functionals is not Borel; in \cite{Kau00} he improves this result showing that such a renorming can be made even Fr\'echet smooth. Following Kaufman and using his deep abstract lemmas, O.~Kurka proves in his paper \cite{Kur11} that given any separable non-reflexive Banach space $X$ and an ordinal $\alpha<\omega_1$, there exists an equivalent strictly convex norm on $X$ such that the corresponding set of norm-attaining functionals is not of the (additive) Borel class $\alpha$ (it is observed in \cite{Kau91} that when $X$ is strictly convex, $NA(X)$ is Borel). 

It seems likely that Kurka's method could be (at least in some cases) used to prove a version of Lemma~\ref{T:Xgen} with weaker assumptions on the descriptive quality of $A$.
\end{enumerate}
\end{remark}

\begin{ack}
I would like to thank Bernardo Cascales (who suggested the topic of this article during my visit in Murcia), Richard Aron and Mat\'ias Raja for their helpful comments, moral support and interest in my work. Most of all, I am very grateful to Ond\v{r}ej Kalenda for many fruitful discussions and careful reading of the manuscript which resulted in further helpful suggestions. 
\end{ack}

\bibliographystyle{plain}
\bibliography{references}

\end{document}